\documentclass{siamltex}[final]
\usepackage{amsmath,amsfonts,amssymb,mathtools}
\usepackage{subfigure}
\usepackage{hyperref}
\usepackage{xcolor}

\newcommand{\norm}[1]{\left\lVert#1\right\rVert}
\newtheorem{remark}{Remark}
\newcommand{\bu}{\mathbf u}
\newcommand{\bv}{\mathbf v}
\newcommand{\cE}{\mathcal E}
\newcommand{\cP}{\mathcal P}
\newcommand{\cI}{\mathcal I}
\newcommand{\R}{\mathbb{R}}
\newcommand{\rd}{\mathrm{d}}
\newcommand{\pr}{\partial_r}

\begin{document}
\title{Error analysis for a Finite Element Discretization of a corotational harmonic map heat flow problem }
\author{Nam Anh Nguyen \thanks{Institut f\"ur Geometrie und Praktische  Mathematik, RWTH-Aachen
University, D-52056 Aachen, Germany (nguyen@igpm.rwth-aachen.de)} \and Arnold Reusken\thanks{Institut f\"ur Geometrie und Praktische  Mathematik, RWTH-Aachen
University, D-52056 Aachen, Germany (reusken@igpm.rwth-aachen.de)} 
}
\maketitle
\begin{abstract}
	We consider the harmonic map heat flow problem for a corotational case. For discretization of this problem we apply a  $H^1$-conforming finite element method in space  combined with a semi-implicit Euler time stepping. The semi-implicit Euler method results in a linear problem in each time step. We restrict to the regime of smooth solutions of the continuous problem and present an error analysis of this discretization method. This results in optimal order discretization error bounds. Key ingredients of the analysis are a discrete energy estimate, that mimics the energy dissipation of the continuous solution, and   a convexity property that  is essential for discrete  stability and for control of the linearization error. We also present numerical results  that validate the theoretical ones.
\end{abstract}

\section{Introduction}
Let $\Omega \subset \mathbb{R}^N$, $N=2,3$,  be a bounded Lipschitz domain and $S^2$ the unit sphere in $\mathbb{R}^3$. The harmonic  map heat flow (HMHF) problem is as follows. Given an initial condition $\bu_0: \Omega \to S^2$, determine $\bu(\cdot,t): \Omega \to S^2$ such that
\begin{equation} \label{HMHFeq}
 \partial_t\mathbf{u}= \Delta \mathbf{u} + \rvert \nabla \mathbf{u} \lvert^2 \mathbf{u}, \quad \bu(\cdot,0)= \bu_0, \quad \bu(\cdot,t)_{|\partial \Omega}=(\bu_0)_{|\partial \Omega}, \quad t \in (0,T]. 
\end{equation}
This problem is obtained as the $L^2$ gradient flow of the Dirichlet energy
\begin{equation} \label{Energy} 
  E(\bv):= \tfrac12 \int_\Omega |\nabla \bv|^2 \, dx
\end{equation}
 of vector fields $\bv: \Omega \to \mathbb{R}^3$ that satisfy a pointwise unit length constraint. Unit length minimizers of this Dirichlet energy are called harmonic maps. This problem is closely related to the Landau-Lifshitz-Gilbert (LLG) equation. The HMHF equation can be considered as the limit of the LLG equation where the precessional term vanishes and only damping is left \cite{Melcher2011}. 
 Harmonic maps, HMHF and LLG equations have numerous practical applications, for example, in the modeling of ferromagnetic materials or of liquid crystals, cf. e.g.~\cite{Lakshmanan2011,LiuLi2013Skyr,prohl2001,Heinze2013}.
 
 There is an extensive mathematical literature in which topics related to well-posedness, weak formulations, regularity, blow-up phenomena and convergence of solutions of the HMHF problem to harmonic maps are studied, cf. e.g~\cite{Struwe85,Struwe08,Eells1964HarmonicMO,Hamilton1975HarmonicMO,KCChang89,Freire95,ChangDingYe1992,vandHout2001}.
 
 Early work on the development and analysis of numerical methods for HMHF or LLG problems is found in \cite{BarPro06,BaLuPr09,prohl2001,BartelsProhl2007,Alouges2008FEMLL, AlougesJaisson06,Cimrak05}. In recent years, there has been a renewed interest in the numerical analysis of methods for this problem class \cite{KovacsLuebich,BaKoWa22,BBPS24,AkBaPa23-pre,ABRW25-pre,NguyenReusken}.
 
In this paper, we study the HMHF problem for a specific \emph{corotational} case. Assume $\Omega$ is the unit disk in $\R^2$ and assume the solution $\bu$  to be corotational. Using polar coordinates on the disk a special type of solution  of \eqref{HMHFeq} is given by
 \begin{equation} \label{polcoord} 
 \mathbf{u}(r,\psi,t) =
   \begin{pmatrix}  \cos{\psi} \sin{ u(t,r) } \\  \sin{\psi} \sin{ u(t,r) } \\ \cos{ u(t,r)}\end{pmatrix}, 
\end{equation}
with a scalar unknown  function $u=u(t,r)$ on $[0,T] \times [0,1]$.  
Note that due to the structure  of the solution $\bu$ in \eqref{polcoord} the unit length constraint is satisfied. 
This leads to the following symmetric HMHF problem.
Given $0 < T <\infty$, $I:=[0,1]$ and $u_0 \in C^2(I)$ with $u_0(0)=u_0(1)=0$ and $|u_0| < \pi$,
determine $u=u(t,r)$ such that
\begin{equation} \label{eq:SSHMHF_Eq}
 \begin{split}
    \partial_tu &= \partial_{rr}u + \frac{1}{r}\partial_r u - \frac{\sin{(2u)}}{2r^2} \quad \text{for } r \in I, t \in (0,T], \\
    u(0,r) &= u_0(r) \quad \text{for } r \in I, \\
    u(t,0) & = u(t,1) =0 \quad\text{for }  t \in [0,T].
\end{split}
\end{equation}
Equation \eqref{eq:SSHMHF_Eq} has been used in models of nematic liquid crystals \cite{vandHout2001}. More importantly, 
the corotational case \eqref{eq:SSHMHF_Eq} plays a fundamental role in the analysis of HMHF.
In the seminal works \cite{ChangDingYe1992,Chang1991} the authors study  finite time singularities of \eqref{HMHFeq}
for the two-dimensional case $N=2$. In these studies the corotational case  \eqref{eq:SSHMHF_Eq} 
plays a crucial role. It is shown that for this problem with initial data $u_0$ with $|u_0| \leq \pi$ a unique global smooth solution exists, whereas for the case with  $|u_0(1)|> \pi$ the solution blows up in finite time, meaning that for the solution $u$ the derivative at $x=0$ becomes arbitrary large: $\lim_{ x\downarrow 0} |u_x(x,t)| \to \infty$ for $t \uparrow T_{\rm crit}$. 
The work \cite{ChangDingYe1992} has motivated further investigations of the blow-up behavior, e.g., \cite{Bertsch_DalPasso_vanderHout_2002,Topp02} where infinitely many solutions of \eqref{eq:SSHMHF_Eq} are constructed whose energy is  bounded by the initial energy for all times $t$, but can \emph{in}crease at certain points $t$, even for a smooth initial condition. 
A further related topic is the analysis  of the blow up rate. Several theoretical aspects of blow up rates of solutions of \eqref{eq:SSHMHF_Eq} are studied in  \cite{vdBergHulshofKing2003,AngHulsMatano09,RaphaelSchweyer13}.

In \cite{ChangDingYe1992} it is shown that  a function $\bu$ of the form \eqref{polcoord} solves \eqref{HMHFeq} if and only if $u$  solves the reduced problem \eqref{eq:SSHMHF_Eq}. Furthermore, the analysis in that paper proves blow-up of $u$, provided $|u_0(1)| > \pi$ in the initial condition holds. Hence, the discretization  method \eqref{eq:SSHMHF_FullyDiscretized} for the reduced problem \eqref{eq:SSHMHF_Eq}, that we analyze in this paper, can be used for numerical studies of the blow-up behavior of solutions of \eqref{HMHFeq}. 

In a series of works, Gustafson \emph{et al.} \cite{GUAN20091,Gustafson_Nakanishi_Tsai_2010,Gustafson2017GlobalSF} investigate the well-posedness and regularity of the $m$-equivariant version of \eqref{eq:SSHMHF_Eq}. Hocquet \cite{Hocquet19} has studied the finite-time singularity of a stochastic version of \eqref{eq:SSHMHF_Eq}.

The energy $E(\cdot)$ in \eqref{Energy} can be rewritten in terms of $u$. We have $E(\bu)=2 \pi \cE(u)$ with
 \begin{equation}  \cE(v)  := \tfrac12 \int_0^1 \left( (\partial_r v)^2 + \frac{\sin^2 v}{r^2} \right) r \mathrm{d}r.  \label{eq:SSHMHF_Energy}
 \end{equation} 
 
There are only very few papers in which numerical aspects  of \eqref{eq:SSHMHF_Eq} are treated. 
The paper \cite{HaynesHuangZegeling2013} treats a moving mesh ansatz for the discretization of \eqref{eq:SSHMHF_Eq}, based on  finite differences  in space combined with an ODE solver, that is used for capturing the blow up behavior of the solution with their numerical method. 
In \cite{NguyenReusken}, a first error analysis of a finite difference discretization of \eqref{eq:SSHMHF_Eq} is presented in the regime of smooth solutions. The main tool to establish stability is the use of $M$-matrix theory. An optimal discretization error bound (apart from a logarithmic term) is derived.

The main contribution of this paper is an error analysis for a finite element discretization of \eqref{eq:SSHMHF_Eq}. 
We restrict to the regime of smooth solutions of the continuous problem.
We consider  an $H^1$-conforming finite element discretization in space combined with a semi-implicit Euler method in time. The semi-implicit Euler method results in a linear problem in each time step.  
We summarize the main ingredients of our analysis. The energy $\cE(\cdot)$ given in \eqref{eq:SSHMHF_Energy} is \emph{not} convex. There, however, is a hidden convexity in the following sense. We use an energy splitting of the form  
\[
   \cE(v) = \tfrac12 \norm{v}_{H^1_r}^2  - \int_0^1 \frac{F\big(v(r)\big)}{r^2} r \, \mathrm{d}r, 
\]
where $\|\cdot\|_{H^1_r}$ is a weighted $H^1$-norm, cf. \eqref{h8} below. It turns out that the function $F$ used in this splitting \emph{is convex}. As far as we know, this convexity property and energy relation have not been discussed in any other work on equation \eqref{eq:SSHMHF_Eq} so far. We also 
show that, for functions $v$ with sufficiently small energy (as made precise in Lemma~\ref{lemma:Bounded_Energy_stationary}), one can bound $\|v\|_{L^\infty(I)}$ by an explicit function of $\cE(v)$.
Both the convexity property and this relation between $\|\cdot \|_{L^\infty}$  and $\cE(\cdot)$ are essential for deriving a discrete energy dissipation property (Theorem \ref{theo:Energy_dissipation}) and discrete stability (Corollary \ref{cor:Discrete_Bound}).
The derivation of a sharp discretization error bound, cf. Theorem \ref{theo:Discr_err}, relies on the classical splitting of the error using the Galerkin projection. Due to conformity of the finite element space, we can use  standard interpolation error bounds. The challenging part is to deal with the linearization error which requires a $L^\infty$-bound on the Galerkin projection of the continuous solution \emph{scaled} by $1/r$, cf. Lemma \ref{Lemhelpbound}. Combining this bound with a suitable convexity argument, we are able to control the linearization error in a satisfactory way.

The method and the error analysis in this work can be extended to the higher-degree corotational harmonic map heat flow considered in \cite{Gustafson2017GlobalSF}. 
However, we do not see how to extend our analysis to the equivariant case \cite{GUAN20091} or to the (equivariant case of) LLG equation \cite{BarPro06,VandenBergWilliams13}. A main obstacle is that  in these problems we do not have the convex structure and a corresponding  energy splitting that are crucial in our analysis.

The remainder of the paper is organized as follows. In Section \ref{sec:VarForm}, we introduce  the variational formulation and derive a convexity property.   In Section~\ref{sec:discr},  we formulate the finite element discretization and derive discrete energy dissipation and discrete stability results. A sharp bound for the discretization error is derived in Section~\ref{sec:error}. We validate our theoretical findings in Section~\ref{sec:NumExp} with a numerical example that demonstrates  convergence rates and energy dissipation. The source code of the numerical experiment can be found in \cite{nguyen_2025_15481333} and is based on the software package \emph{Netgen/NGSolve}.

\section{Variational formulation} \label{sec:VarForm}
For a  finite element discretization of this problem, we need a suitable variational formulation, that we now introduce. For $u,v \in H^1(I)$, we define the scalar products 
\begin{equation} \label{h8}   \begin{split} (v,w)_{0,r}& :=\int_0^1 v(r)w(r) r \, \rd r, \quad  (v,w)_{1,r}:= (\pr v, \pr w)_{0,r}, \\
      (v,w)_{H^1_r} & :=(v,w)_{1,r}+\left(\frac{v}{r},\frac{w}{r}\right)_{0,r}.
     \end{split}
 \end{equation}
Corresponding norms are denoted by $\|\cdot\|_{0,r}$, $\|\cdot\|_{1,r}$ and $\|\cdot \|_{H_r^1}$, respectively.
The use of the scaling with $r$ in these scalar products is natural because the problem \eqref{eq:SSHMHF_Eq} originates from a transformation to polar coordinates. We define the spaces
\begin{align*}
	L^2_r &:= \left\{v:[0,1]\rightarrow \mathbb{R}: \norm{v}_{0,r}< \infty\right\}, \\
    H_{r,0}^1 &:= \left\{v \in L^2_r: \norm{v}_{H^1_r}< \infty, v(0)=v(1)=0\right\}.
\end{align*}
The solution $u$ of \eqref{eq:SSHMHF_Eq}  solves the variational problem
\begin{align}
    \left(\partial_tu, v \right)_{0,r} + \left(u,v\right)_{H^1_r} = \left(\frac{2u-\sin{(2u)}}{2r^2}, v \right)_{0,r}
    \quad \text{for all}~ t \in [0,T], ~v \in H_{r,0}^1. \label{eq:SSHMHF_VariationalForm}
\end{align}
Furthermore, a smooth solution of \eqref{eq:SSHMHF_VariationalForm} solves \eqref{eq:SSHMHF_Eq}. In the remainder, we consider the variational formulation \eqref{eq:SSHMHF_VariationalForm}, which will be the basis of a finite element discretization. 

We briefly comment on the interpretation of \eqref{eq:SSHMHF_VariationalForm} as an $L^2_r$ gradient flow problem. Using the definition \eqref{eq:SSHMHF_Energy}, we obtain that the $L^2_r$ gradient flow
\[
  (\partial_t u, v)_{0,r} = - \left(\cE'(u),v\right)_{0,r}, \quad \cE'(u)=\frac{\partial \cE}{\partial u},
\]
coincides with the variational problem \eqref{eq:SSHMHF_VariationalForm}.
Taking $v=\partial_t u$ in this gradient flow equation we obtain the  energy decay property 
\begin{equation} \label{Edecay}
\cE(u(t)) + \int_0^t \|\partial_t u\|_{0,r}^2 \, \rd s \leq \cE(u_0). \end{equation}
We derive a few further properties that will be used in the analysis of the finite element discretization in the following sections. For this we introduce 
\begin{align}
      F(v)&:=\tfrac12 \left(v^2 - \sin^2 v\right), \\
    f(v) &:= F'(v) = v- \tfrac12 \sin (2 v), 
\end{align}
for which the fundamental relation
\begin{equation} 
   \cE(u) = \tfrac12 \norm{u}_{H^1_r}^2  - \int_0^1 \frac{F\big(u(r)\big)}{r^2} r \mathrm{d}r \label{eq:SSHMHF_Rewrite_Energy}
\end{equation}
holds. 
The result in the following lemma, which is derived using a technique introduced in  \cite[Lemma 2.3]{RoxanasPhD2017},   bounds the maximum norm  $\|\cdot\|_\infty=\|\cdot\|_{L^\infty(I)}$ of a function by its energy. We define, for $b >0$,
\begin{align*}
    \mathcal{E}_{b} := \{u \in H^1_{r,0}~|~ \cE(u) \leq 2b\}. 
\end{align*}
\begin{lemma} 
    For all $u \in \mathcal{E}_{1}$ the following holds:
    \begin{align*}
         \|u\|_\infty \leq 2 \arcsin \sqrt{\tfrac12 \cE(u)} .
    \end{align*}
    \label{lemma:Bounded_Energy_stationary}
\end{lemma}
\begin{proof} Define
\begin{align*}
    G(z) :=\pi \int_0^z |\sin(s)|\mathrm{d}s.
\end{align*}
Note that
\[
    G(-z) = \pi \int_0^{-z} |\sin(s)| \rd s = -G(z), \quad
    G'(z) = |\sin(z)|,
\]
and thus $G$ is odd and strictly increasing in $[-\pi,\pi]$ with a strictly monotonic inverse. Take $\alpha \in (0,\tfrac12 ]$. Using $G(2\alpha \pi)=\pi(1- \cos (2\alpha \pi))$ it follows that $G:\, [- 2\alpha \pi, 2\alpha \pi] \to [-\pi(1- \cos (2\alpha \pi)),  \pi (1 -\cos (2\alpha \pi))]$ is bijective. With $\beta=\tfrac12(1-\cos(2\alpha \pi)) \in (0,1]$ and using the identity $1-\cos(2\alpha \pi) = 2 \sin^2(\alpha \pi)$, this can be rewritten as 
\begin{equation} \label{hh} G: \, [-2\arcsin\sqrt{\beta}, 2\arcsin\sqrt{\beta}] \to [-2 \pi \beta, 2 \pi \beta].  
\end{equation}
Take $u \in \cE_b$. 
Then with $u(0)=0$ we obtain, for any $r \in I$,
\begin{align*}
    |G(u(r))| &= \left|\int_0^r \partial_{\Tilde{r}} G(u(\Tilde{r}))\, \mathrm{d}\Tilde{r} \right| = \pi \left| \int_0^r |\sin(u)| \partial_{\Tilde{r}} u\,  \mathrm{d}\Tilde{r} \right| \\
    &\leq \frac{\pi}{2}\int_0^1 \left( (\partial_r u) ^2 +\frac{\sin^2 u}{r^2} \right) r \, \mathrm{d}r = \pi \cE (u)  . 
\end{align*}
Using $\cE(u)\leq 2$, we obtain $-2\pi \leq G(u(r))\leq 2 \pi$. The result in \eqref{hh} with $\beta=\tfrac12  \cE(u)$ and the existence of a strictly increasing monotonic inverse of $G$ yield the estimate $|u(r)| \leq 2 \arcsin \sqrt{\beta} = 2\arcsin\sqrt{\tfrac12 \cE(u)}$ for all $r \in I$. This completes the proof.
\end{proof}

Below we will use the following direct corollary of the result above:
\begin{equation} \label{maxestimate}
  u \in \cE_1 \quad \Rightarrow \quad \|u\|_\infty \leq \pi.  
\end{equation}
It will be useful to have relations between the energy $\cE(u)$ and the norm $\norm{u}^2_{H^1_r}$. 
These are given in the following lemma.
\begin{lemma}
    \label{theo:Energy_Norm_Equivalence}
    The following estimates hold:
    \begin{align} \label{upperB}
     \cE(u) & \leq \tfrac12 \norm{u}^2_{H^1_r} \quad \text{for all}~u \in H_{r,0}^1, \\
     C_b \norm{u}^2_{H^1_r} & \leq \cE(u) \quad \text{for all}~ u \in \cE_b, ~ 0<b<1,  \label{lowerB}
    \end{align}
with $C_b:= \frac{b(1-b)^2}{2 \, \arcsin^2 \sqrt{b} }$.
\end{lemma}
\begin{proof}
    The estimate  \eqref{upperB} follows from
   $
        |\sin z| \leq |z|$ 
    for all $z \in \mathbb{R}$. 
   Take $u \in \cE_b, ~ 0<b<1$. From Lemma~\ref{lemma:Bounded_Energy_stationary}  we obtain $\|u\|_\infty \leq 2 \arcsin \sqrt{b} =:z_b < \pi$. Note $\sin z_b= 2\sqrt{b}(1-b) >0$. The line through $(0,0)$ and $(z_b, \sin z_b)$ has slope $\alpha_b=\frac{\sqrt{b}(1-b)}{\arcsin \sqrt{b}}$, and for all $z \in [0,z_b]$ we have $\sin z \geq \alpha_b z$. Using this we obtain for $r \in [0,1]$
   \[
    \sin^2 u(r) = \sin^2 |u(r)| \geq \alpha_b^2 u(r)^2,
   \]
which  implies the result \eqref{lowerB}. 
\end{proof}

From the results \eqref{upperB}-\eqref{lowerB} it follows that we have an equivalence $\|\cdot\|_{H_r^1}^2 \sim \cE(\cdot)$ on $\cE_b$, provided $b<1$. 

The nonlinear part of \eqref{eq:SSHMHF_VariationalForm} is given by $\left(\frac{2u-\sin (2u)}{2 r^2}, v\right)_{0,r}=\left(\frac{f(u)}{r^2}, v\right)_{0,r}$. Note that due to the $r^2$ in the denominator there is a blow up effect for $r \downarrow 0$.  In the analysis below, to control this term we will use a convexity argument. The energy $\cE(\cdot)$ is \emph{not} convex, but the function $F$, cf. the relation \eqref{eq:SSHMHF_Rewrite_Energy}, is convex, since $F''(z)=f'(z)=1-\cos (2z) \geq 0$ for all $z \in \R$. This will be used to derive the results in the next lemma, which play a key role in finite element error analysis.
\begin{lemma} \label{lemcontrol}
 The following holds:
 \begin{align}
  f(z_1)(z_2-z_1) & \leq F(z_2) -F(z_1) \quad \text{for all}~z_i \in \R \label{control1} \\
  \left(\frac{f(u)}{r^2}, w-u\right)_{0,r} & \leq \tfrac12 \big(\|w\|_{H_r^1}^2 -\|u\|_{H_r^1}^2\big) + \cE(u)-\cE(w) \quad \text{for all}~u,w \in H_{0,r}^1. \label{control2}
 \end{align}
\end{lemma}
\begin{proof}
The result \eqref{control1} is due to the convexity of $F$ and $F'=f$.  
For $u,w \in H_{0,r}^1$ we have, using the estimate \eqref{control1} and the fundamental relation \eqref{eq:SSHMHF_Rewrite_Energy},
\begin{align*}
 \left(\frac{f(u)}{r^2}, w-u\right)_{0,r} & = \int_0^1 \frac{f\big(u(r)\big)}{r^2} ( w(r)-u(r)) r \, \rd r\\
  & \leq \int_0^1 \frac{F\big(w(r)\big) - F\big( u(r)\big)}{r^2} \, r \, \rd r  \\
  & = \cE(u) -\cE(w) + \tfrac12 \big(\|w\|_{H_r^1}^2 -\|u\|_{H_r^1}^2\big),
\end{align*}
which proves the result \eqref{control2}. 
\end{proof}

A further useful estimate that we use in the error analysis is:
\begin{align}
    \label{Inequ_nonlin}
    |f(z_1)-f(z_2)| \leq \left(z_1^2 + z_2^2\right) |z_1-z_2| \quad \text{for all}~z_i \in \R,
\end{align}
which follows from
\begin{align*}
 |f(z_1)-f(z_2)| &= \left| \int_0^1 f'(z_2+s(z_1-z_2))\, \rd s  \right| |z_1-z_2|  \\ & = 2 \int_0^1 \sin^2(z_2+s(z_1-z_2))\, \rd s \, |z_1-z_2| \\
 & \leq 2 \int_0^1 (z_2+s(z_1-z_2))^2\, \rd s \, |z_1-z_2|  \\ & = \tfrac23 (z_1^2+ z_1 z_2 + z_2^2) |z_1-z_2| \leq \left(z_1^2 + z_2^2\right) |z_1-z_2|. 
\end{align*}

\section{Discrete problem}\label{sec:discr}
The discrete problem that we consider uses an implicit Euler method for time discretization, where the nonlinear term in \eqref{eq:SSHMHF_VariationalForm} is treated in a semi-implicit way. Due to this, in each time step there is a linear problem to be solved. For the discretization in space we use a standard  finite element space. 

We introduce some further notation. We use the notation $H_0^1(I)$ for the standard Sobolev space with homogeneous boundary conditions. 
In space we use grid points $r_i = i h, i=0,\hdots, N$, with $h=1/N$ for some $N \in \mathbb{N}$. We use a uniform grid to simplify the presentation. All results can easily be generalized to the case of a quasi-uniform grid. We use the $H^1_0(I)$-conforming finite element space
\begin{align*}
    S^k_h & := \{v \in C^0(I)~|~ v_h|_{[r_i,r_{i+1}]} \in P_k \text{ for all } i=0,\hdots, N-1 \}, \quad k \geq 1,  
\\
    S_{h,0}^k &:=  S^k_h \cap H_0^1(I). 
\end{align*}
For discretization in time we use a fixed time step $\tau$, with $\tau J=T$ for some $J \in \mathbb{N}$.
We assume a \emph{given} initial value $u_h^0 \in  S_{h,0}^k$ that approximates $u_0$. The choice of this approximation will be discussed below. We introduce the following \emph{discrete problem}: For $j\geq 0$, determine $u_h^{j+1} \in S^k_{h,0}$ such that 
\begin{align}
    \left(\frac{u^{j+1}_h-u^j_h}{\tau}, v_h\right)_{0,r} + \left(u_h^{j+1},v_h\right)_{H^1_r} = \left(\frac{f(u_h^j)}{r^2},v_h\right)_{0,r} \quad \text{for all}~ v_h \in S^k_{h,0}.
    \label{eq:SSHMHF_FullyDiscretized}
\end{align}
Thus in each time step we have a uniquely solvable \emph{linear} problem. 
\subsection{Stability analysis}
We derive a stability result for this discrete problem.
\begin{theorem} \label{theo:Energy_dissipation}
 For the solution $(u_h^j)_{1 \leq j \leq J}$ of \eqref{eq:SSHMHF_FullyDiscretized} the following holds:
\begin{align}
 \cE(u_h^{j+1}) +\tfrac{1}{\tau}\norm{u_h^{j+1}-u_h^j}_{0,r}^2  + \norm{u_h^{j+1}-u_h^j}_{H^1_r}^2 & \leq \cE(u_h^j), \label{Eestimate1}\\
        \cE(u_h^{j+1}) & \leq \cE(u_h^{j}). \label{theo:SSHMHF_Energy_Inequ}
    \end{align}
    \label{theo:Energy_Dissipation}
\end{theorem}
\begin{proof}
    We test \eqref{eq:SSHMHF_FullyDiscretized} with $v_h = u_h^{j+1}-u_h^j$, which yields,
    \begin{align}
        &\tfrac{1}{\tau}\norm{u_h^{j+1}-u_h^j}_{0,r}^2 + \tfrac{1}{2}\left( \norm{u_h^{j+1}}_{H^1_r}^2 - \norm{u_h^{j}}_{H^1_r}^2 + \norm{u_h^{j+1}-u_h^j}_{H^1_r}^2  \right) \notag\\
        &\quad \quad = \left(\frac{f(u_h^j)}{r^2}, u_h^{j+1}-u_h^j\right)_{0,r}.  \label{eq:SSHMHF_Stability_test}
    \end{align}
The estimate \eqref{control2} yields
    \begin{align*}
      \left(\frac{f(u_h^j)}{r^2}, u_h^{j+1}-u_h^j\right)_{0,r} \leq \tfrac12 \big(\norm{ u_h^{j+1}}_{H^1_r}^2 - \norm{ u_h^{j}}_{H^1_r}^2\big) +\cE(u_h^j) - \cE(u_h^{j+1}).
    \end{align*}
Using this in \eqref{eq:SSHMHF_Stability_test} yields the result \eqref{Eestimate1}. The result in \eqref{theo:SSHMHF_Energy_Inequ} is a direct consequence of \eqref{Eestimate1}.
\end{proof}

The result \eqref{Eestimate1} is a discrete analogue of \eqref{Edecay}.
\begin{corollary} \label{cor:Discrete_Bound}
If $u_h^0 \in \mathcal{E}_{1}$, then 
\begin{align*}
    \|u_h^{j}\|_{L^\infty} \leq \pi, \quad 0 \leq j \leq J,
\end{align*}
holds.
\end{corollary}
\begin{proof}
    This follows immediately from Theorem \ref{theo:Energy_Dissipation} and  \eqref{maxestimate}. 
\end{proof}

\section{Discretization error analysis}\label{sec:error}
In this section we derive bounds for the discretization error of the scheme \eqref{eq:SSHMHF_FullyDiscretized}.
\subsection{Preliminaries}
We collect some results that will be used in the error analysis in Section~\ref{sectdiscrerror}. 
\begin{lemma} \label{lembasic} The following holds:
\begin{equation}
    \norm{v}_{0,r} \leq \norm{v}_{L^2} \leq \|v\|_{H_r^1} \leq \sqrt{2} \norm{v}_{H^1}, \quad v \in H^1(I) \quad \text{with}~v(0)=0. \label{eq:H1_embedding}
\end{equation}
\end{lemma}
\begin{proof}
The first two inequalities in \eqref{eq:H1_embedding} follow directly from the definitions. For the third one we use $\int_0^1  v'(r)^2 r \, \mathrm{d}r 
     \leq  \int_0^1 v'(r)^2 \, \mathrm{d}r$ and the (Hardy-type) inequality
\[
 \int_0^1 \frac{v(r)^2}{r}\, \mathrm{d}r  
     \leq  \int_0^1 v'(r)^2 \, \mathrm{d}r,
\]
which follows from
\[
  \frac{v(r)^2}{r} = \frac{1}{r} \left( \int_0^r v'(s) \,\mathrm{d}s \right)^2\leq \int_0^r v'(s)^2 \,\mathrm{d}s  \leq \int_0^1 v'(s)^2 \,\mathrm{d}s 
\]
and integrating this inequality.
\end{proof}

Hence, we have the embeddings $S^k_{h,0} \subset H^1_{0} \subset H^1_{r,0}$.
\\
The  nodal interpolation operator on  $S^k_h$ is denoted by $\cI_h$. 
By classical finite element theory, e.g.\cite[Chapter 4]{BrennerScott08}, we have the interpolation error bounds 
\begin{align}
    \norm{v-\mathcal{I}_h v}_{H^1} &\leq c h^m \norm{v}_{H^{m+1}}, \quad v \in H^{m+1}(I),~~0 \leq m \leq k, \label{eq:Interpolation_error_H1} \\
   \norm{v-\mathcal{I}_h v}_{W^1_\infty} &\leq c h^{m} \norm{v}_{W^{m+1}_\infty},\quad v \in W_\infty^{m+1}(I), ~~0 \leq m \leq k. \label{eq:Interpolation_error_Linf}
\end{align} 
We define the Galerkin projection $\mathcal{P}_h:H^1_{r,0} \rightarrow  S^k_{h,0}$ by
\begin{align}
    \left(\mathcal{P}_h v,v_h\right)_{H^1_r} = \left(v,v_h\right)_{H^1_r} \quad \text{for all}~v_h \in S^k_{h,0}. \label{eq:Ritz_Projection_Def}
\end{align}
From the projection property, \eqref{eq:H1_embedding}  and \eqref{eq:Interpolation_error_H1},  it follows that
\begin{align}
    \norm{\mathcal{P}_hv}_{H^1_r} & \leq  \norm{v}_{H^1_r},  \quad v \in H_{r,0}^1, \label{eq:Proj_Stab_H1r} \\
 \norm{v-\mathcal{P}_hv}_{H^1_r} &\leq  c\norm{v-\mathcal{I}_hv}_{H^1} \leq c h^m \norm{v}_{H^{m+1}},\quad v \in H_{r,0}^1 \cap H^{m+1}(I),  \label{eq:Projection_Error_H1r}
    \end{align}
for $0 \leq m \leq k$. 
In the analysis we need a bound for $\|\frac{v_h}{r}\|_{L^\infty}:=\max_{r \in [0,1]} \left| \frac{v_h(r)}{r} \right|$ in terms of $\norm{v_h}_{H^1_r}$ for finite element functions  $v_h \in S^k_{h,0}$.
We derive such a result.
\begin{lemma}\label{Leminverse}
    For  $v_h \in S^k_{h,0}$ the following estimate holds:
    \begin{align}
    \norm{\frac{v_h}{r}}_{L^\infty} & \leq ch^{-1} \norm{v_h}_{H^1_r}. \label{inverse}
    \end{align}
\end{lemma}
\begin{proof}
	First we consider $ \norm{\frac{v_h}{r}}_{L^\infty([0,h])}$.
	Using the transformation $r \mapsto \hat{r}=\frac{1}{h}r$, $\hat{v}_h(\hat{r}):=v_h(h\hat{r}) = v_h(r)$ and $v_h'(r)=\frac{d}{dr} \hat{v}_h(\hat{r}) = \frac{1}{h}\hat{v}_h'(\hat{r})$ we obtain
	\begin{align*}
		\int_0^h v_h'(r)^2 r \, \mathrm{d}r = \int_0^1 \hat{v}_h'(\hat{r})^2\hat{r} \, \mathrm{d}\hat{r}.
	\end{align*}
	With a suitable norm equivalence constant $c_k$ that depends only on the polynomial degree $k$ we have
	\begin{align*}
		\norm{\frac{\hat{v}_h}{\hat{r}}}_{L^\infty([0,1])} \leq c_k \left(\int_0^1\hat{v}_h'(\hat{r})^2\hat{r} \, \mathrm{d}\hat{r} \right)^{\frac{1}{2}}
	\end{align*}
	and it follows that
	\begin{align}
		\label{eq:Linf_Ineq_0_h}
		\begin{split}
		 \norm{\frac{v_h}{r}}_{L^\infty([0,h])} &= h^{-1}	\norm{\frac{\hat{v}_h}{\hat{r}}}_{L^\infty([0,1])} \leq c_k h^{-1} \left(\int_0^1\hat{v}_h'(\hat{r})^2\hat{r} \, \mathrm{d}\hat{r} \right)^{\frac{1}{2}} \\
		 &= c_k h^{-1} \left(\int_0^h v_h'(r)^2 r \, \mathrm{d}r\right)^{\frac{1}{2}} \leq c_k h^{-1} \norm{v_h}_{H^1_r}.
	 \end{split}
	\end{align}
	Now assume that $r \in [h,1]$. From $\frac{v_h(r)}{r} = \frac{v_h(h)}{r}$ + $\frac{1}{r}\int_h^r v_h'(s) \, \mathrm{d}s$ we have
	\begin{align}
		\left|\frac{v_h(r)}{r} \right| \leq \norm{\frac{v_h}{r}}_{L^\infty([0,h])} + \frac{1}{r}\left|\int_h^r v_h'(s)\, \mathrm{d}s\right|. \label{eq:Linf_Ineq_h_1}
	\end{align}
	For the second term on the right hand side we obtain
	\begin{align*}
		\frac{1}{r}\left|\int_h^r v_h'(s)\, \mathrm{d}s\right| &=\frac{1}{r}\left|\int_h^r v_h'(s)s^{\frac{1}{2}}s^{-\frac{1}{2}}\, \mathrm{d}s\right| \\
		&\leq \left(\int_h^r v_h'(s)^2 s \, \mathrm{d}s\right)^{\frac{1}{2}} \frac{1}{r} \left(\int_h^r \frac{1}{s} \, \mathrm{d}s \right)^{\frac{1}{2}} \\
		&\leq \norm{v_h}_{H^1_r} h^{-1} \frac{h}{r} \left(\ln\left(\frac{r}{h}\right)\right)^{\frac{1}{2}} \\
		&\leq c h^{-1} \norm{v_h}_{H^1_r}.
	\end{align*}
	Using this in \eqref{eq:Linf_Ineq_h_1} and combining with \eqref{eq:Linf_Ineq_0_h} completes the proof.
\end{proof}

\begin{lemma} \label{Lemhelpbound} For $w \in H^{k+1}(I), k\geq 1$, with $w(0)=0$ the following estimate holds:
 \begin{equation} \label{stabB} 
\norm{\frac{\mathcal{P}_h w(r)}{r}}_{L^\infty} \leq c  \norm{w}_{H^{k+1}(I)}. 
\end{equation}
\end{lemma}
\begin{proof}
We start with a triangle inequality
\begin{equation} \label{AA}
\norm{\frac{\mathcal{P}_h w(r)}{r}}_{L^\infty} \leq \norm{\frac{\mathcal{P}_h w - \cI_h w}{r}}_{L^\infty} + \norm{\frac{\mathcal{I}_h w - w}{r}}_{L^\infty} + \norm{\frac{ w}{r}}_{L^\infty}.
\end{equation}
With $w(0)=0$ we obtain
\begin{equation} \label{U1}
 \norm{\frac{ w}{r}}_{L^\infty}= \norm{\frac{\int_0^r  w'(s) \, ds}{r}}_{L^\infty} \leq \|w\|_{W_\infty^1(I)}.
\end{equation}
Using the same argument (note: $(\mathcal{I}_h w - w)(0)=0$) and \eqref{eq:Interpolation_error_Linf} we obtain
\begin{equation} \label{U2}
 \norm{\frac{\mathcal{I}_h w -  w}{r}}_{L^\infty} \leq \norm{\mathcal{I}_h w -  w}_{W_\infty^1(I)} \leq c \|w\|_{W_\infty^1(I)}.
\end{equation}
Using \eqref{inverse} and \eqref{eq:Interpolation_error_H1} we obtain
\begin{equation} \label{U3} \begin{split} 
\norm{\frac{\mathcal{P}_h w - \cI_h w}{r}}_{L^\infty} & \leq \frac{c}{h} \norm{\mathcal{P}_h w - \cI_h w}_{H^1_r}  \\
& \leq  \frac{c}{h}\norm{w - \cI_h w}_{H_r^1} \\ &  \leq   c h^{k-1} \|w\|_{H^{k+1}}.
\end{split}
\end{equation}
We use the estimates \eqref{U1}, \eqref{U2} and \eqref{U3} in \eqref{AA} and with the Sobolev embedding $ H^2(I) \hookrightarrow W_\infty^1(I)$ we obtain the result \eqref{stabB}. 
\end{proof}
\ \\

\subsection{Discretization error bound} \label{sectdiscrerror}
We define the error $e^{j}_h :=u_h^j-u(t_j)$ where $u$ solves \eqref{eq:SSHMHF_Eq} and $u_h^{j+1}$ solves \eqref{eq:SSHMHF_FullyDiscretized}. We use the notation $u(t)=u(t, \cdot)$. 
From now on, we assume that the solution $u=u(t) = u(t,r)$ of \eqref{eq:SSHMHF_Eq}  satisfies the  regularity assumption
\begin{align}
    \label{eq:SSHMHF_Regularity_Exact_Solution}
    \max_{t \in [0,T]} \left\{\norm{u(t)}_{W^{k+1}_\infty} + \norm{\partial_t u(t)}_{H^1}+\norm{\partial_{tt} u(t)}_{L^2} \right\} \leq c < \infty.
\end{align} 
\begin{theorem}
	\label{theo:Discr_err}
    Given  $u_0 \in H^{k+1}(I)$ with $u_0 \in \cE_1$, take $u_h^0 \in S_{h,0}^k \cap \cE_1$. For the error $e^{j}_h :=u_h^j-u(t_j)$, $1 \leq j \leq J$, the following holds:
    \begin{align} \label{main1}
        \norm{\partial_r e^{j}_h}_{0,r} \leq  C (\tau + h^k) +c \|u_0-u_h^0\|_{H_r^1}, 
    \end{align}
    where the constants $C$, $c$ are independent of $h$, $\tau$ but depend on $T$ and on the regularity assumption \eqref{eq:SSHMHF_Regularity_Exact_Solution}. Assume $\|u_0-u_h^0\|_{H_r^1} \leq c h^{k}$. For $h$ and $\tau$  sufficiently small  we have 
    \begin{align} \label{main2}
        \norm{e^{j}_h}_{H^1_r} \leq C (\tau + h^k).
    \end{align}
\end{theorem}
\begin{proof}
In the proof we use $c$ to denote  a varying constant independent of $h$ and $\tau$.  We split the error in the usual way
\[
  e_h^j=\big(u_h^j-P_hu(t_j)\big) + \big(P_hu(t_j)- u(t_j)\big) =: \tilde e_h^j +\hat e_h^j. 
\]
Using \eqref{eq:Projection_Error_H1r} we have bounds for the projection error $\hat e_h^j$ and obtain
\begin{align} 
 \|\partial_r e_h^j\|_{0,r} & \leq  \|\partial_r  \tilde e_h^j\|_{0,r} + c h^k,  \label{eq:Fully_discretized_Error_splitting} \\
\|e_h^j\|_{H_r^1} & \leq  \| \tilde e_h^j\|_{H_r^1} + c h^k. \label{eq:Fully_discretized_Error_splittingB}
\end{align}
For bounding the terms with $\tilde e_h^j$ we need a rather technical analysis. To improve the presentation we first outline the key ingredients of the proof below:
\begin{itemize}
\item We use the continuous problem and a canonical test function $v_h= \frac{1}{\tau}(\tilde e_h^{j+1}-\tilde e_h^j)$ to derive the recursive relation \eqref{eq:Fully_Discr_Error_Eq}.
\item The terms $R_1$, $R_2$ are bounded using standard consistency arguments. For the terms $R_3,R_4,R_5$ we need linearization. We use \eqref{Inequ_nonlin} and Lemma~\ref{Lemhelpbound}. For $R_5$ we use a convexity argument. This then results in the recursive relation \eqref{eqtotal}.
\item With the fundamental relation \eqref{eq:SSHMHF_Rewrite_Energy} this yields the recursive relation \eqref{F4} for the energy $\cE(\tilde e_h^j)$, which can be treated using standard arguments.
\item Using $\|\partial_r v\|_{0,r} \leq \sqrt{2 \cE(v)}$ and Lemma~\ref{theo:Energy_Norm_Equivalence} in combination with \eqref{eq:Fully_discretized_Error_splitting}-\eqref{eq:Fully_discretized_Error_splittingB} we obtain the error bounds in \eqref{main1} and \eqref{main2}. 
\end{itemize}
We now continue our proof. 
Using the differential equation \eqref{eq:SSHMHF_VariationalForm} and the projection property  \eqref{eq:Ritz_Projection_Def}, we obtain  for arbitrary $v_h \in S^k_{h,0}$ the relation
\begin{align*}
    &\frac{1}{\tau}\big(\mathcal{P}_h (u(t_{j+1})-u(t_j)),v_h\big)_{0,r} + \big(\mathcal{P}_h u(t_{j+1}),v_h\big)_{H^1_r} \\
    &\quad =\left( \frac{1}{\tau}\mathcal{P}_h\big(u(t_{j+1})-u(t_j)\big)-\frac{1}{\tau}\big(u(t_{j+1})-u(t_j)\big),v_h\right)_{0,r} \\
    &\quad+ \left(\frac{1}{\tau}\big(u(t_{j+1})-u(t_j)\big)-\partial_t u(t_{j+1}),v_h\right)_{0,r} \\
    &\quad+\left(\frac{f(u(t_{j+1}))-f(u(t_{j}))}{r^2},v_h\right)_{0,r} +\left(\frac{f(u(t_j))-f(\mathcal{P}_h u(t_{j}))}{r^2},v_h\right)_{0,r} \\
    &\quad + \left(\frac{f(\mathcal{P}_h u(t_{j}))}{r^2},v_h\right)_{0,r}.
\end{align*}
We subtract this from \eqref{eq:SSHMHF_FullyDiscretized}, test with $v_h = \tfrac{1}{\tau}(\Tilde{e}_h^{j+1}-\Tilde{e}_h^j)$ and use Young's inequality. Thus we get, with arbitrary $\varepsilon>0$,
\begin{align}
    &\norm{\frac{\Tilde{e}_h^{j+1}-\Tilde{e}_h^j}{\tau}}_{0,r}^2 + \frac{1}{2\tau} \left(\norm{\Tilde{e}^{j+1}_h}^2_{H^1_r}-\norm{\Tilde{e}^{j}_h}^2_{H^1_r}+\norm{\Tilde{e}^{j+1}_h-\Tilde{e}^{j}_h}^2_{H^1_r}\right) \notag\\ 
    &\quad \leq 2\varepsilon\norm{\frac{\Tilde{e}_h^{j+1}-\Tilde{e}_h^j}{\tau}}_{0,r}^2 +\frac{1}{\varepsilon}\norm{\frac{1}{\tau}\mathcal{P}_h\big(u(t_{j+1})-u(t_j)\big)-\frac{1}{\tau}\big(u(t_{j+1})-u(t_j)\big)}_{0,r}^2 \notag\\
    &\quad\quad\quad+ \frac{1}{\varepsilon}\norm{\partial_t u(t_{j+1})-\frac{1}{\tau}\big(u(t_{j+1})-u(t_j)\big)}_{0,r}^2 \notag \\
    &\quad\quad\quad+\left(\frac{f(u(t_{j}))-f(u(t_{j+1}))}{r^2},\frac{\Tilde{e}_h^{j+1}-\Tilde{e}_h^j}{\tau}\right)_{0,r} \notag\\
    &\quad\quad\quad+\left(\frac{f(\mathcal{P}_h u(t_{j}))-f(u(t_j))}{r^2},\frac{\Tilde{e}_h^{j+1}-\Tilde{e}_h^j}{\tau}\right)_{0,r}\notag \\
    &\quad\quad\quad + \left(\frac{f(u_h^j)-f(\mathcal{P}_h u(t_{j}))}{r^2},\frac{\Tilde{e}_h^{j+1}-\Tilde{e}_h^j}{\tau}\right)_{0,r} \notag\\
    & =:2\varepsilon\norm{\frac{\Tilde{e}_h^{j+1}-\Tilde{e}_h^j}{\tau}}_{0,r}^2 + \tfrac{1}{\varepsilon}R_1 +\tfrac{1}{\varepsilon}R_2 + R_3+R_4+R_5. \label{eq:Fully_Discr_Error_Eq}
\end{align}
For the term $R_1$ we obtain, using $\|v\|_{0,r} \leq \|v\|_{H_r^1}$,  
\begin{align}
  R_1^\frac12   & \leq \norm{\frac{1}{\tau}\mathcal{P}_h\big(u(t_{j+1})-u(t_j)\big)-\frac{1}{\tau}\big(u(t_{j+1})-u(t_j)\big)}_{H^1_r} \notag\\
    & \leq \frac{1}{\tau} \int_{t_j}^{t_{j+1}} \norm{ \cP_h \partial_t u(s) -\partial_t u(s)}_{H_r^1}\, \rd s \notag\\ 
    & \overset{\eqref{eq:Projection_Error_H1r}}{\leq} ch^{k} \norm{\partial_t u}_{L^\infty([0,T];H^{k+1}(I))}
\label{boundR1}
\end{align}
For $R_2$ we have
\begin{equation} \label{boundR2}
  R_2^\frac12  = \left\| \frac{1}{\tau}\int_{t_{j}}^{t_{j+1}} (s-t_{j})\partial_{tt}u(s)\,  \rd s\right\|_{0,r} \leq
  \tfrac12 \tau \max_{s \in [0,T]}\norm{\partial_{tt}u(s)}_{0,r}.
\end{equation}
For the term $R_3$ we note  
\begin{equation}
  \left|\frac{u(t,r)}{r}\right| = \frac{1}{r} \left| \int_0^r \partial_r u(t,s)\, \rd s\right| \leq \|u\|_{L^\infty([0,T];C^1(I))}.  \label{eq:Linf_uoverR_bound}
\end{equation}
Using this and \eqref{Inequ_nonlin} we obtain:
\begin{align}
    R_3&=\left(\frac{f(u(t_j))-f(u(t_{j+1}))}{r^2},\frac{\tilde e_h^{j+1}-\tilde e_h^j}{\tau}\right)_{0,r} \notag \\ 
    & \quad \leq \left(\frac{u(t_j)^2+u(t_{j+1})^2}{r^2}|u(t_{j+1})-u(t_j)|, \left| \frac{\tilde e_h^{j+1}- \tilde e_h^j}{\tau} \right|\right)_{0,r} \notag \\
    &\quad \leq 2 \norm{u}_{L^\infty([0,T]; C^1(I))}^2\norm{u(t_{j+1})-u(t_j)}_{0,r} \norm{\frac{\tilde e_h^{j+1}-\tilde e_h^j}{\tau}}_{0,r} \notag \\
    & \quad \leq   \tau 2 \norm{u}_{L^\infty([0,T]; C^1(I))}^2 \norm{u}_{C^1([0,T]; L^2_{r})}  \norm{\frac{\tilde e_h^{j+1}- \tilde e_h^j}{\tau}}_{0,r}  \notag \\
    & \quad \leq   \varepsilon \norm{\frac{\tilde e_h^{j+1}- \tilde e_h^j}{\tau}}_{0,r}^2   +  \frac{c_u}{\varepsilon} \tau^2,  \label{boundR3}
\end{align}
with $c_u= \norm{u}_{L^\infty([0,T]; C^1(I))}^4 \norm{u}_{C^1([0,T]; L^2_{r})}^2$.  

We now consider $R_4$.
We use \eqref{Inequ_nonlin}, \eqref{eq:Linf_uoverR_bound} and Lemma \ref{Lemhelpbound} and obtain
\begin{align}
    R_4&=\left(\frac{f(u(t_j))-f(\mathcal{P}_h u(t_{j}))}{r^2},\frac{\Tilde{e}_h^{j+1}-\Tilde{e}_h^j}{\tau}\right)_{0,r} \notag \\
    &\quad\leq c\left(\frac{u(t_j)^2+(\mathcal{P}_hu(t_j))^2}{r}\left|\frac{u(t_j)-\mathcal{P}_hu(t_j)}{r}\right|,\left|\frac{\Tilde{e}_h^{j+1}-\Tilde{e}_h^j}{\tau}\right|\right)_{0,r} \notag \\
    &\quad \leq \frac{c}{\varepsilon}\norm{\frac{u(t_j)-\mathcal{P}_hu(t_j)}{r}}_{0,r}^2 + \varepsilon \norm{\frac{\Tilde{e}_h^{j+1}-\Tilde{e}_h^j}{\tau}}_{0,r}^2 \notag \\
    &\quad \leq \frac{c}{\varepsilon}\norm{u(t_j)-\mathcal{P}_hu(t_j)}_{H^1_r}^2 + \varepsilon \norm{\frac{\Tilde{e}_h^{j+1}-\Tilde{e}_h^j}{\tau}}_{0,r}^2 \notag \\
    &\quad \overset{\eqref{eq:Projection_Error_H1r}}{\leq} \frac{c}{\varepsilon}h^{2k} + \varepsilon \norm{\frac{\Tilde{e}_h^{j+1}-\Tilde{e}_h^j}{\tau}}_{0,r}^2 \label{boundR4}.
\end{align}

For the term $R_5$ we need  the following elementary estimate 
\begin{align}
    \big| \sin(2x)+\sin(2y) -\sin(2x+2y)  \big| &= 2 \big| \sin(2x) \sin(y)^2+ \sin(2y) \sin(x)^2  \big| \notag \\
    &= 4  \big|\cos(x)\sin(x) \sin(y)^2 + \cos(y) \sin(y) \sin(x)^2 \big|\notag\\
    &= 4\big| \sin(x) \sin(y) \left(\cos(x)\sin(y)+\cos(y)\sin(x)\right)\big| \notag \\
    &= 4 \big|\sin(x) \sin(y) \sin(x+y)\big| \notag \\
    &\leq 4 |x| |y| |\sin(x+y) |.
    \label{eq:Trigonometric_identity}
\end{align}
We use this estimate and the convexity property \eqref{control1}
\begin{align}
   & R_5 =\left(\frac{f(u_h^j)-f(\cP_h u(t_{j}))}{r^2},\frac{\tilde e_h^{j+1}-\tilde e_h^j}{\tau}\right)_{0,r}  = \left(\frac{\tilde e_h^j-\frac{1}{2} \sin(2 \tilde e_h^j)}{r^2},\frac{\tilde e_h^{j+1}- \tilde e_h^j}{\tau}\right)_{0,r}  \notag\\
    &\qquad \quad-\frac{1}{2} \left(\frac{\sin(2u_h^j)+\sin(-2 \cP_h u(t_j))-\sin(2 \tilde e_h^j)}{r^2},\frac{\tilde  e_h^{j+1}- \tilde e_h^j}{\tau}\right)_{0,r} \notag  \\
    &  \leq \frac{1}{\tau} \int_0^1 \frac{F(\tilde e_h^{j+1})-F(\tilde e_h^j)}{r^2} r \, \mathrm{d}r + 2 \left( \frac{|u_h^j|| \cP_h u(t_j)||\sin(\tilde e_h^j)|}{r^2}, \frac{\left| \tilde e_h ^{j+1}- \tilde e_h^j\right|}{\tau} \right)_{0,r}. \label{boundR5}
\end{align}
We use the uniform boundness result of Corollary~\ref{cor:Discrete_Bound} and Lemma~\ref{Lemhelpbound} to get
\begin{equation}
\label{boundR5_case1}
	\begin{split}
& 2\left( \frac{|u_h^j|| \cP_h u(t_j)||\sin(\tilde e_h^j)|}{r^2}, \frac{\left| \tilde e_h ^{j+1}- \tilde e_h^j\right|}{\tau} \right)_{0,r} \\ &\leq  \frac{4}{\varepsilon}\norm{u_h^j}_{L^\infty}^2 \norm{\frac{\mathcal{P}_hu(t_j)}{r}}_{L^\infty}^2 \norm{\frac{\sin(\Tilde{e}_h^j)}{r}}_{0,r}^2 + \varepsilon \norm{\frac{\tilde{e}^{j+1}_h-\tilde{e}^{j}_h}{\tau}}_{0,r}^2  \\
	&\leq \frac{c}{\varepsilon}\norm{\frac{\sin(\Tilde{e}_h^j)}{r}}_{0,r}^2  +\varepsilon \norm{\frac{\tilde{e}^{j+1}_h-\tilde{e}^{j}_h}{\tau}}_{0,r}^2.
	\end{split}
\end{equation} 
We collect \eqref{boundR1}, \eqref{boundR2}, \eqref{boundR3}, \eqref{boundR4} and \eqref{boundR5_case1} in \eqref{boundR5} to obtain from \eqref{eq:Fully_Discr_Error_Eq}
\begin{equation} \label{eqtotal} \begin{split}
    &\norm{\frac{\Tilde{e}_h^{j+1}-\Tilde{e}_h^j}{\tau}}_{0,r}^2 + \frac{1}{2\tau} \left(\norm{\Tilde{e}^{j+1}_h}^2_{H^1_r}-\norm{\Tilde{e}^{j}_h}^2_{H^1_r}+\norm{\Tilde{e}^{j+1}_h-\Tilde{e}^{j}_h}^2_{H^1_r}\right)  \\ 
    &\quad \leq 5\varepsilon\norm{ \frac{\Tilde{e}_h^{j+1}-\Tilde{e}_h^j}{\tau}}_{0,r}^2+ \frac{c}{\varepsilon}\left( \tau^2+h^{2k}  \right)  \\
    &\quad \quad +\frac{1}{\tau} \int_0^1 \frac{F(\Tilde{e}_h^{j+1})-F(\Tilde{e}_h^j)}{r^2} r \, \mathrm{d}r +  \frac{c}{\varepsilon}\norm{\frac{\sin(\Tilde{e}_h^j)}{r}}_{0,r}^2.
\end{split}
\end{equation}
Now choose $\varepsilon>0$ small enough such that $\norm{ \frac{\Tilde{e}_h^{j+1}-\Tilde{e}_h^j}{\tau}}_{0,r}^2$ can be absorbed to the left hand side. Using the fundamental relation \eqref{eq:SSHMHF_Rewrite_Energy} we arrive at
\[
    \frac{1}{\tau}\left( \cE(\Tilde{e}_h^{j+1}) -\cE(\Tilde{e}_h^j)\right) \leq \frac{c}{\varepsilon} \left(  \tau^2+h^{2k}+\cE(\Tilde{e}_h^j) \right).
\]
We multiply  by $\tau$ and thus get
\begin{equation} \label{F4}
  \cE(\Tilde{e}_h^{j}) \leq (1+c_1 \tau) \cE(\Tilde{e}_h^{j-1}) +c_2 \tau (\tau^2 + h^{2k}).
\end{equation}
Recursive application and using \eqref{upperB} we obtain
\begin{align*}
    \cE(\Tilde{e}_h^{j})  & \leq c_2 \left(\sum_{i=0}^{j-1} (1+c_1 \tau)^i\right)\tau (\tau^2 +h^{2k})  +(1+c_1 \tau)^j \cE(\tilde e_h^0) \\
    & \leq c_2 e^{c_1 T} T (\tau^2 + h^{2k}) + \tfrac12 e^{c_1 T} \|u_h^0-\cP_h u_0\|_{H_r^1}^2. 
\end{align*} 
We apply the triangle inequality $\|u_h^0-\cP_h u_0\|_{H_r^1} \leq \|u_h^0- u_0\|_{H_r^1}+ \|u_0-\cP_h u_0\|_{H_r^1}$. For the last term we use \eqref{eq:Projection_Error_H1r} and then absorb it in the $h^{2k}$ term. This yields
\[
  \sqrt{\cE(\Tilde{e}_h^{j})} \leq  c (\tau + h^k) + c \|u_h^0- u_0\|_{H_r^1}.
\]
Now use $\|\partial_r \Tilde{e}_h^{j}\|_{0,r} \leq \sqrt{2 \cE(\Tilde{e}_h^{j})}$ and 
combine this with the triangle inequality in \eqref{eq:Fully_discretized_Error_splitting}. This yields the result \eqref{main1}.
If $\|u_h^0-\cP_h u_0\|_{H_r^1} \leq c h^k$ holds, we have $\cE({\tilde e}_h^{j}) \leq c (\tau + h^k)$. For $h$ and $\tau$ sufficiently small we thus have $\cE({\tilde e}_h^{j}) \leq K < 2$, i.e, $ \tilde e_h^j \in \mathcal{E}_{b}$ with $b <1$. Using the result \eqref{lowerB} and \eqref{eq:Fully_discretized_Error_splittingB} we obtain
\[
 \|e_h^j\|_{H_r^1} \leq \|\tilde e_h^j\|_{H_r^1} + c h^k \leq c \sqrt{\cE(\tilde e_h^j)} + h^k \leq c (\tau + h^k),
\]
which proves the result \eqref{main2}.
\end{proof}

\section{Numerical results} \label{sec:NumExp}
We consider a problem as in \eqref{eq:SSHMHF_Eq} with $u_0(r) = \pi(1-r)r$ and $T=0.1$.  In this case we have a globally smooth solution. We apply the method \eqref{eq:SSHMHF_FullyDiscretized} and determine the errors at the end time point, i.e. $\|u_{\rm ref}^J-u^J\|_{H^1_r}$ and  $\|u_{\rm ref}^J-u^J\|_{0,r}$. The source code of the experiments can be found in \cite{nguyen_2025_15481333}.
\begin{remark} \rm
	A sufficiently accurate reference solution $u_{\rm ref}$ is determined by using the scheme \eqref{eq:SSHMHF_FullyDiscretized} with sufficiently small mesh and time step sizes. The accuracy is validated by comparing numerical solutions with those of a BDF2 version of \eqref{eq:SSHMHF_FullyDiscretized}.
\end{remark}

Results are presented in Tables~\ref{table:SSHMHF_p1_conv}, \ref{table:SSHMHF_p2_conv} and \ref{table:SSHMHF_tau_conv}. In the tables~\ref{table:SSHMHF_p1_conv} and \ref{table:SSHMHF_p2_conv} we take a very small time step $\tau$ and measure convergence for linear and quadratic finite elements with decreasing mesh size $h$. In Table~\ref{table:SSHMHF_tau_conv} we take a very fine mesh size $h$ and measure convergence for decreasing time step size $\tau$. In all cases we observe  optimal order of convergence in the norm $\|\cdot\|_{H_r^1}$, as predicted by Theorem \ref{theo:Discr_err}. We also see that in  the first two tables the convergence in $\norm{\cdot}_{0,r}$ is one order higher than in $\norm{\cdot}_{H^1_r}$, which is expected, but not covered by our theoretcal analysis. Figure \ref{fig:plot_energy} shows the energy dissipation of the numerical solution which agrees with Theorem \ref{theo:Energy_Dissipation}.
\begin{table}[ht!]
	\centering
	\begin{tabular}{ |p{2.5cm}|| p{2.65cm}| p{1cm}|| p{2.65cm}| p{1cm}|}
		\hline
		$\tau = 10^{-6}$ & $\| u_{\text{ref}}^{M}-u^M\|_{0,r}$  & EOC & $\| u_{\text{ref}}^{M}-u^M\|_{H^1_{r}}$  & EOC  \\
		\hline
		$h=2^{-1}$ & $4.7141\cdot 10^{-2}$ & $-$ & $1.9893 \cdot 10^{-1}$ & $-$ \\
		$h=2^{-2}$ & $1.2293\cdot 10^{-2}$ & $1.94$ & $7.6974 \cdot 10^{-2}$ & $1.37$ \\
		$h=2^{-3}$ & $3.1080\cdot 10^{-3}$ & $1.98$ & $3.5357\cdot 10^{-2}$ & $1.12$ \\
		$h=2^{-4}$ & $7.7945\cdot 10^{-4}$ & $2.00$ & $1.7277\cdot 10^{-2}$ & $1.03$ \\
		$h=2^{-5}$ & $1.9530\cdot 10^{-4}$ & $2.00$ & $8.5880\cdot 10^{-3}$ & $1.01$ \\
		\hline 
	\end{tabular}
	\caption{Discretization error for \eqref{eq:SSHMHF_FullyDiscretized} with $S^1_{h,0}$; time step $\tau$ fixed. }
	\label{table:SSHMHF_p1_conv}
\end{table}
\begin{table}[ht!]
	\centering
	\begin{tabular}{ |p{2.5cm}|| p{2.65cm}| p{1cm}|| p{2.65cm}| p{1cm}|}
		\hline
		$\tau = 10^{-6}$ & $\| u_{\text{ref}}^{M}-u^M\|_{0,r}$  & EOC & $\| u_{\text{ref}}^{M}-u^M\|_{H^1_{r}}$  & EOC  \\
		\hline
		$h=2^{-1}$ & $2.3366\cdot 10^{-3}$ & $-$ & $3.5584 \cdot 10^{-2}$ & $-$ \\
		$h=2^{-2}$ & $2.8004\cdot 10^{-4}$ & $3.06$ & $8.5863 \cdot 10^{-2}$ & $2.05$ \\
		$h=2^{-3}$ & $3.4426\cdot 10^{-5}$ & $3.02$ & $2.1358\cdot 10^{-3}$ & $2.01$ \\
		$h=2^{-4}$ & $4.3833\cdot 10^{-6}$ & $2.97$ & $5.3363\cdot 10^{-4}$ & $2.00$ \\
		$h=2^{-5}$ & $1.2500\cdot 10^{-6}$ & $1.81$ & $1.3348\cdot 10^{-4}$ & $1.99$ \\
		\hline 
	\end{tabular}
	\caption{Discretization error for \eqref{eq:SSHMHF_FullyDiscretized} with $S^2_{h,0}$; time step $\tau$ fixed. }
	\label{table:SSHMHF_p2_conv}
\end{table}
\begin{table}[ht!]
	\centering
	\begin{tabular}{ |p{2.5cm}|| p{2.65cm}| p{1cm}|| p{2.65cm}| p{1cm}|}
		\hline
		$h = 2^{-14}$ & $\| u_{\text{ref}}^{M}-u^M\|_{0,r}$  & EOC & $\| u_{\text{ref}}^{M}-u^M\|_{H^1_{r}}$  & EOC  \\
		\hline
		$\tau=1.25 \cdot 10^{-2}$ & $1.3783\cdot 10^{-2}$ & $-$ & $5.2826 \cdot 10^{-2}$ & $-$ \\
		$\tau=6.25 \cdot 10^{-3}$ & $7.0245\cdot 10^{-3}$ & $0.97$ & $2.6921 \cdot 10^{-2}$ & $0.97$ \\
		$\tau=3.125 \cdot 10^{-3}$ & $3.5467\cdot 10^{-3}$ & $0.99$ & $1.3592 \cdot 10^{-2}$ & $0.99$ \\
		$\tau=1.5625 \cdot 10^{-3}$ & $1.7821\cdot 10^{-3}$ & $0.99$ & $6.8300 \cdot 10^{-3}$ & $0.99$ \\
		$\tau=7.8125 \cdot 10^{-4}$ & $8.9328\cdot 10^{-4}$ & $1.00$ & $3.4233 \cdot 10^{-3}$ & $1.00$ \\
		\hline 
	\end{tabular}
	\caption{Discretization error for \eqref{eq:SSHMHF_FullyDiscretized}; mesh size $h$ fixed. }
	\label{table:SSHMHF_tau_conv}
\end{table}
\begin{figure}[ht!]
	\centering
	\includegraphics[scale=0.16]{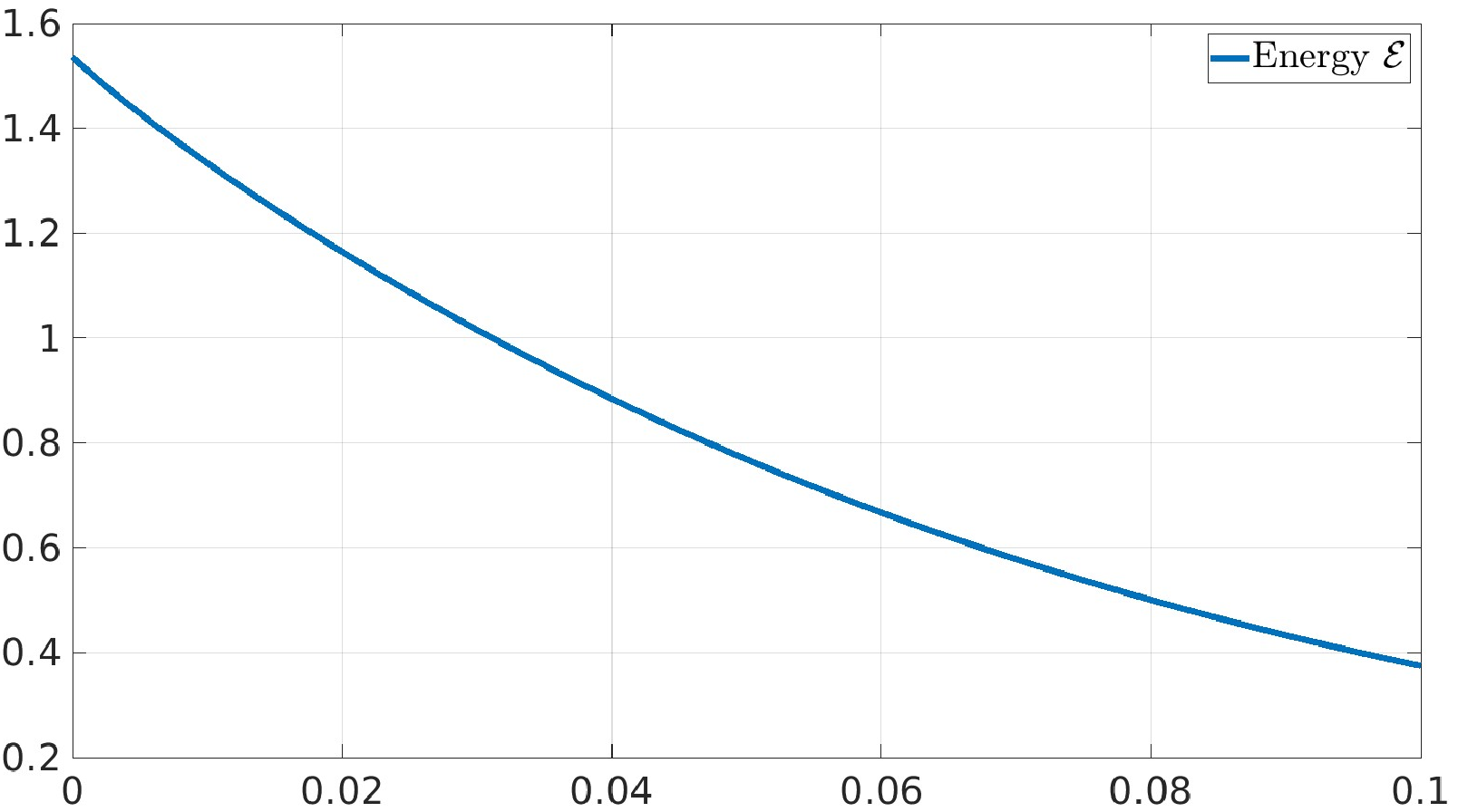}
	\caption{Energy $\cE(u^j_h)$ of discrete solution for $ h=2^{-6}, \tau = 10^{-6}$.}
	\label{fig:plot_energy}
\end{figure}
\ \\[10ex]
{\bf Acknowledgements} The authors acknowledge funding by the Deutsche Forschungsgemeinschaft (DFG, German Research Foundation) – project number 442047500 – through the Collaborative Research Center “Sparsity and Singular
Structures” (SFB 1481).
\newpage
\bibliographystyle{siam}
\bibliography{literature}{}

\end{document}